\documentclass[12pt]{amsart}
\usepackage{amsfonts, amssymb, latexsym, hyperref, xcolor, tikz,tikz-cd,transparent}
 \usepackage{euler, amsfonts, amssymb, latexsym, epsfig, epic}
\usepackage{ytableau}

\newtheorem{theorem}{Theorem}[section]

\newtheorem{proposition}[theorem]{Proposition}

\newtheorem*{claim*}{Claim}
\newtheorem{corollary}[theorem]{Corollary}
\newtheorem{Main Conjecture}[theorem]{Main Conjecture}
\newtheorem{conjecture}[theorem]{Conjecture}
\newtheorem{problem}[theorem]{Problem}
\theoremstyle{definition}

\theoremstyle{remark}

\newtheorem{example}[theorem]{Example}

\theoremstyle{plain}

\newcommand{\cellsize}{12}
\newlength{\cellsz} 
\newsavebox{\cell}
\sbox{\cell}{\begin{picture}(\cellsize,\cellsize)
\put(0,0){\line(1,0){\cellsize}}
\put(0,0){\line(0,1){\cellsize}}
\put(\cellsize,0){\line(0,1){\cellsize}}
\put(0,\cellsize){\line(1,0){\cellsize}}
\end{picture}}
\newcommand\cellify[1]{\def\thearg{#1}\def\nothing{}%
\ifx\thearg\nothing
\vrule width0pt height\cellsz depth0pt\else
\hbox to 0pt{\usebox{\cell} \hss}\fi%
\vbox to \cellsz{
\vss
\hbox to \cellsz{\hss$#1$\hss}
\vss}}
\newcommand\tableau[1]{\vtop{\let\\\cr
\baselineskip -16000pt \lineskiplimit 16000pt \lineskip 0pt
\ialign{&\cellify{##}\cr#1\crcr}}}

\newcommand{\excise}[1]{}

\begin{document}
\pagestyle{plain}
\title{Castelnuovo-Mumford regularity and Schubert geometry}
\author{Alexander Yong}
\address{Dept.~of Mathematics, U.~Illinois at Urbana-Champaign, Urbana, IL 61801, USA} 
\email{ayong@illinois.edu}
\date{February 12, 2022}

\maketitle
\begin{abstract}
We study  the Castelnuovo-Mumford regularity of tangent cones of Schubert varieties. Conjectures about this statistic are presented; these are proved for the \emph{covexillary} case. This builds on work of L.~Li and the author 
on these tangent cones,
as well as that of J.~Rajchgot-Y.~Ren-C.~Robichaux-A.~St.~Dizier-A.~Weigandt and 
J.~Rajchgot-C.~Robichaux-A.~Weigandt on tableau rules for computing regularity of some matrix Schubert varieties.
\end{abstract}

\section{Introduction}

Let $GL_n/B$ be the \emph{complete flag variety}; $GL_n$ is the group of $n\times n$ invertible complex matrices
and $B$ is the Borel subgroup of invertible upper triangular matrices. $B$ acts with finitely many orbits
$X_w^{\circ}=BwB/B\cong {\mathbb C}^{\ell(w)}$;  
$w\in {\mathfrak S}_n=$ the symmetric group on $[n]:=\{1,2,\ldots,n\}$ and $\ell(w)$ is the \emph{Coxeter length} of $w$, that is, $\ell(w)=\#\{i<j: w(i)>w(j)\}$. Their closures 
\[X_w:=\overline{X_w^{\circ}}=\coprod_{v\leq w} X_v^{\circ}\] 
are the \emph{Schubert varieties}; here $v\leq w$ refers to \emph{(strong) Bruhat order}. Let $T\subset GL_n$ be the maximal torus of 
invertible diagonal matrices. The $T$-fixed points are $e_v:=vB/B$. To study the local structure of $X_w$, it suffices to study only
the points $e_v$ (for $v\leq w$), since $B$ provides local isomorphisms to any other point of $X_{v}^\circ\subseteq X_w$. A book reference is \cite{Fulton:YT}.

Let $({\mathcal O}_{p,Y},{\mathfrak m}_p, {\Bbbk})$ be the local ring of a point $p$ in a variety $Y$. The \emph{associated graded ring} \cite[Chapter 10]{Atiyah} with respect to the ${\mathfrak m}_p$-adic filtration is
\[R_{p,Y}:={\rm gr}_{{\mathfrak m}_p}{\mathcal O}_{p,Y}=\bigoplus_{i=0}^{\infty} {\mathfrak m}_p^{i}/{\mathfrak m}_{p}^{i+1} \ \ \  \ \ ({\mathfrak m}_p^0:={\mathcal O}_{p,Y}).\]
$R_{p,Y}$ has a ${\mathbb Z}$-graded \emph{Poincar\'e series}
\begin{equation}
\label{eqn:firstPS}
{\sf PS}_{p,Y}(q)=\sum_{i=0}^{\infty} \dim({\mathfrak m}_p^{i}/{\mathfrak m}_{p}^{i+1}) q^i=\frac{H_{p,Y}(q)}{(1-q)^{\dim(Y)}},
\end{equation}
where $H_{p,Y}(q)\in {\mathbb Z}[q]$. $H_{p,Y}(1)$ is the \emph{Hilbert-Samuel multiplicity}. In the case 
$p=e_v$ and $Y=X_w$,
let ${\sf PS}_{v,w}(q)=P_{p,Y}(q), R_{v,w}=R_{p,Y}$, and $H_{v,w}(q)=H_{p,Y}(q)$.

We study the \emph{Castelnuovo-Mumford regularity} ${\rm Reg}(R_{v,w})$, viewed as a graded module over 
$\Bbbk[{\mathfrak m}_{e_v}/{\mathfrak m}_{e_v}^2]$. This statistic measures, in some sense, the ``complexity'' of $R_{v,w}$; see Section~\ref{sec:3} for definitions. Outside of Schubert geometry, study of regularity of the associated graded ring appears in, \emph{e.g.},
\cite{BrodLinh, Trung} and the references therein.

\begin{conjecture}
\label{conj:1}
${\rm Reg}(R_{v,w})=\deg H_{v,w}(q)$.
\end{conjecture}

\begin{conjecture}[Semicontinuity]
\label{conj:2}
If $u\leq v\leq w$ in Bruhat order then ${\rm Reg}(R_{u,w})\geq {\rm Reg}(R_{v,w})$.
\end{conjecture}

\begin{conjecture}[Upper bound]
\label{conj:3}
${\rm Reg}(R_{v,w})\leq \frac{\ell(w)-\ell(v)-1}{2}$.
\end{conjecture}

Proposition~\ref{evenstronger} shows they follow from 
earlier conjectures with L.~Li \cite{LiYong1, LiYong2}; see Section~\ref{sec:5}.
Conjectures~\ref{conj:1} and~\ref{conj:2} imply that ${\rm Reg}(R_{u,v})$ is a singularity measure that 
falls into the framework of \cite{Governing}. In particular, it would imply the locus of points $p\in X_w$ with 
``${\rm Reg}(p)\geq k$'' is described using \emph{interval pattern avoidance}. 

Speculatively, a strengthening of Conjecture~\ref{conj:3} holds, namely, 
${\rm Reg}(R_{v,w})\leq \deg \, P_{v,w}(q)$ where $P_{v,w}(q)$ is the \emph{Kazhdan-Lusztig
polynomial}; but, the evidence is not strong ($n\leq 6$).

The papers \cite{LiYong1,LiYong2} study the tangent cones in the case $w$ is \emph{covexillary}, \emph{i.e.,} $w$ \emph{avoids} the pattern $3412$ (there are not indices
$i_1<i_2<i_3<i_4$ such that $w(i_1),w(i_2),w(i_3),w(i_4)$ are in the same relative order as $3412$). This defines a subfamily with a number of prior results. For example, \emph{ibid.}~gives formulas for $H_{v,w}(q)$
and related them to the \emph{Kazhdan-Lusztig polynomials}; a combinatorial formula for the latter was already known due to work
of A.~Lascoux \cite{Lascoux}. One also has a ``diagonal Gr\"obner basis theorem'' for \emph{matrix Schubert
varieties} \cite{KMY}.\footnote{Some of these results are stated for \emph{vexillary} rather than covexillary family; this is
a matter of convention.} These results play a role in our work. This is our main result:

\begin{theorem}
\label{thm:main}
Conjectures~\ref{conj:1},~\ref{conj:2}, and~\ref{conj:3} hold if $w$ is covexillary. In this case, there is a combinatorial rule for ${\rm Reg}(R_{v,w})$ (see Theorem~\ref{thm:formula}), and ${\sf Reg}(R_{v,w})=\deg \, P_{v,w}$.
\end{theorem}

Our proof of the first part of Theorem~\ref{thm:main} makes use of \cite{LiYong2}, which degenerates the tangent cone of the \emph{Kazhdan-Lusztig ideal}
${\mathcal N}_{v,w}$ to the Gr\"obner limit \cite{KMY} of the matrix Schubert variety ${\overline X}_{\kappa(v,w)}$ for a \emph{different} covexillary permutation $\kappa(v,w)$. Thereby, $H_{v,w}(q)$ can be expressed in terms of 
\emph{flagged Grothendieck polynomials} \cite{LS:groth, KMY}. We were inspired by the paper of J.~Rajchgot--Y.~Ren--C.~Robichaux--A.~St.~Dizier--A.~Weigandt \cite{Raj}, who determine the degree of a \emph{symmetric Grothendieck polynomial} to find the regularity of ${\overline X}_{w}$ when $w$ is \emph{Grassmannian} (has at most one descent). Recent work of J.~Rajchgot--C.~Robichaux--A.~Weigandt \cite{RRW}  extends that formula to vexillary permutations, which we apply (see also the generalization in O.~Pechenik-D.~Speyer-A.~Weigandt's \cite{PSW}, and E.~Hafner's \cite{Hafner}).

In Section~\ref{sec:2}, we recall the notion of Kazhdan-Lusztig ideals/varieties \cite{Governing}. We also recapitulate necessary
results about its tangent cone from \cite{LiYong1,LiYong2}. We summarize definitions and facts we need about
regularity in Section~\ref{sec:3}. We then prove our main result in Section~\ref{sec:4}. Final remarks are collected in Section~\ref{sec:5}.

\section{Kazhdan-Lusztig varieties}\label{sec:2}

Let $\Omega_v^{\circ}=B_{-}vB/B$ be the
\emph{opposite Schubert cell} where $B_{-}\subset GL_n$ consists of invertible lower triangular matrices. $\Omega_{id}^{\circ}$ is the \emph{opposite big cell}; it is an affine open neighborhood
of $(id)B/B$. Hence $v\Omega_{id}^{\circ}\cap X_w$ is an affine open neighborhood of $X_w$ centered at $e_v$. However, by \cite[Lemma~A.4]{KL},
\begin{equation}
\label{eqn:KLfact}
X_w\cap v\Omega_{id}^{\circ}\cong (X_w\cap \Omega_v^{\circ})\times {\mathbb A}^{\ell(w)}.
\end{equation}
Hence it suffices to study the \emph{Kazhdan-Lusztig variety} ${\mathcal N}_{v,w}:=X_w\cap \Omega_v^{\circ}$.

Explicit coordinates and equations for ${\mathcal N}_{v,w}$ were first studied in work with A.~Woo \cite{Governing}. Let ${\sf Mat}_{n\times n}$ be the set of 
all $n\times n$ complex matrices. The coordinate ring is ${\mathbb C}[{\bf z}]$ where ${\bf z}=\{z_{ij}\}_{i,j=1}^n$ are the
functions on the entries of a generic matrix $Z$. Here $z_{ij}$ corresponds to the entry in the $i$-th row
from the \emph{bottom}, and the $j$-th column to the right. 

Realize $\Omega_{v}^{\circ}$ as a affine subspace of ${\sf Mat}_{n\times n}$ consisting of matrices $Z^{(v)}$ where
$z_{n-v(i)+1,i}=1$, and $z_{n-v(i)+1,s}=0, z_{t,i}=0$ for $s>i$ and $t>n-v(i)+1$. Let ${\bf z}^{(v)}\subseteq {\bf z}$ 
be the unspecialized variables. Furthermore, let $Z_{st}^{(v)}$ be the southwest $s\times t$ submatrix of $Z^{(v)}$. The \emph{rank matrix} 
is
$r^w=(r_{ij}^w)_{i,j=1}^n$
(which we index in the same manner), where 
$r_{ij}^w=\#\{h: w(h)\geq n-i+1,h\leq j\}$. One combinatorial characterization of \emph{Bruhat order} is that $v\leq w$ if and only if
$r_{ij}^v\leq r_{ij}^w$ for all $1\leq i,j\leq n$.

The \emph{Kazhdan-Lusztig ideal} is 
$I_{v,w}\subset {\mathbb C}[z^{(v)}]$ 
generated by all $r_{st}^w+1$ minors
of $Z_{st}^{(v)}$ where $1\leq s,t\leq n$. As explained in \cite{Governing},
${\mathcal N}_{v,w}\cong {\rm Spec}\left({\mathbb C}[z^{(v)}]/I_{v,w}\right)$;
this is reduced and irreducible.

\begin{example}\label{exa:1}
Let $w=7314562$, $v=1423576$ (in one line notation). The rank matrix 
 $r^w$ and the matrix of variables $Z^{(v)}$ are, respectively,
 \[r^w=\left(\begin{matrix}
 1&2&3&4&5&6&7\\
 1&2&2&3&4&5&6\\
 1&2&2&3&4&5&5\\
 1&1&1&2&3&4&4\\
 1&1&1&1&2&3&3\\
 1&1&1&1&1&2&2\\
 1&1&1&1&1&1&1
 \end{matrix}\right), \ \ 
 Z^{(v)}=\left(\begin{matrix}
1 & 0 & 0 & 0 &0 &0 &0\\
z_{61} & 0 & 1 & 0 & 0 &0 &0\\
z_{51} & 0 & z_{53} & 1 & 0 &0 &0\\
z_{41} & 1 & 0 & 0 & 0 & 0 &0\\
z_{31} & z_{32} & z_{33} & z_{34} & 1 &0 &0\\
z_{21} & z_{22} & z_{23} & z_{24} & z_{25} &0 &1\\
z_{11} & z_{12} & z_{13} & z_{14} & z_{15} &1 &0\\
\end{matrix}\right)\]
The Kazhdan-Lusztig ideal $I_{1423576,7314562}$ contains among its generators, all $2\times 2$ minors
of $Z_{25}^{(v)}$ but also inhomogeneous elements such as 
\begin{equation}
\label{eqn:inhomog}
\left|\begin{matrix} z_{51} & 0 & z_{53}\\ z_{41}&1&0\\ z_{31}& z_{32} & z_{33}\end{matrix}\right|=z_{51}z_{33}+z_{53}z_{41}z_{32}-z_{53}z_{31}.
\end{equation}
This generator, \emph{per se}, does not imply $I_{1423576,7314562}$ is inhomogeneous; however one can confirm
the ideal is in fact inhomogeneous with respect to the standard grading using Macaulay2's function \texttt{isHomogeneous}. These ideals (and their statistics) can be computed using \url{https://faculty.math.illinois.edu/~ayong/Schubsingular.v0.2.m2}.\qed
\end{example}

We also need the \emph{Schubert determinantal ideal} $I_w$ which is defined similarly as $I_{v,w}$ except that we 
repace $Z^{(v)}$ with the matrix $Z=(z_{ij})$. The zero-set is the \emph{matrix Schubert variety}.

Given $f\in {\mathbb C}[z^{(v)}]$, let ${\sf LD}(f)$ denote the lowest degree homogeneous component of $f$.
Now, define the \emph{(Kazhdan-Lusztig) tangent cone ideal} to be
\[I_{v,w}'=\langle {\sf LD}(f):f\in I_{v,w}\rangle.\]
\emph{E.g.}, if $f$ is the polynomial in (\ref{eqn:inhomog}) then ${\sf LD}(f)=z_{51}z_{33}-z_{53}z_{31}$.
The \emph{tangent cone} of ${\mathcal N}_{v,w}$ is
\[{{\mathcal N}}_{v,w}':={\rm Spec}\left({\mathbb C}[z^{(v)}]/I_{v,w}'\right).\]
This can be computed using Macaulay2's \texttt{tangentCone} function.

\section{Castelnuovo-Mumford regularity basics}\label{sec:3}

The \emph{Castelnuovo-Mumford regularity} of a finitely generated graded module $M=\bigoplus_{j\in {\mathbb Z}}M^{(j)}$
over a standard ${\mathbb N}$-graded ring $S=\bigoplus_{j\geq 0} S^{(j)}$ is defined by
\[{\rm Reg}(M)=\max\{f_j(M)+j:j\geq 0\}\]
where
\[f_j(M):=\begin{cases}
\sup\{n:H_{S_+}^j(M)_n\neq 0\} & \text{if $H_{S_+}^j(M)\neq 0$,}\\
-\infty & \text{otherwise.}
\end{cases}\]
Here $S_+=\bigoplus_{j>0}S^{(j)}$ is the irrelevant ideal of $S$ and 
$H_{S_+}^i(M)$ is the $i$-th local cohomology module of $M$ with respect to $S_+$ (and its endowed grading).
 We refer the reader to the book \cite[Chapter~15]{Brod} for further details.
One has an expression for the Poincar\'e series
\begin{equation}
\label{eqn:easyreg1}
{\sf PS}_M(q)=\frac{{\mathcal K}_{M}(q)}{(1-q)^{\dim(M)}},
\end{equation}
where ${\mathcal K}_{M}(q)\in {\mathbb Z}[q]$; see, \emph{e.g.}, \cite[Corollary~4.1.8]{Bruns.Herzog}. Let $h_M(q)$ be \emph{Hilbert function}
and $p_M(q)$ be the \emph{Hilbert polynomial}. Hilbert's theorem states that $h_M(q)=p_M(q)$ for all sufficiently large $q$.
The \emph{postulation number} is
\[{\rm post}(M)=\max\{n: h_M(n)\neq p_M(n)\}.\]
By \cite[Proposition~4.1.12]{Bruns.Herzog},
\[{\rm post}(M)=\deg \, {\mathcal K}_{M}(q) -\dim \, M.\]

It is known (and not hard) that when $M$ is Cohen-Macaulay, ${\rm Reg}(M)={\rm post}(M)+\dim \, M$. Hence
\begin{equation}
\label{therelation}
{\rm Reg}(M)=\deg {\mathcal K}_{M}(q).
\end{equation}

Now suppose $S={\mathbb C}[x_1,\ldots,x_N]$ and $M=S/J$ is the $S$ module where $J\subseteq S$ is an ideal that is
standard graded homogeneous. $M=S/J$ has a minimal free resolution
\[0\to \bigoplus_j S(-j)^{\beta_{i,j}(S/J)}\to \bigoplus_j S(-j)^{\beta_{i-1,j}(S/J)}\to \cdots \to \bigoplus_j S(-j)^{\beta_{0,j}(S/J)}\to S/J\to 0.\]
Here $i\leq N$ and $S(-j)$ is the free $S$-module where degrees of $S$ are shifted by $j$.  Also,
\[{\rm Reg}(M):=\max\{j-i: \beta_{i,j}(M)\neq 0\},\]
and
\[{\sf PS}_{S/J}(q)=\frac{{\mathcal K}_{S/J}(q)}{(1-q)^{N}},\]
where ${\mathcal K}(S/J,q)\in {\mathbb Z}[q]$. If $S/J$ is Cohen-Macaulay, (\ref{therelation}) says
\begin{equation}
\label{eqn:easyreg2}
{\rm Reg}(S/J)=\deg {\mathcal K}(S/J,q)-{\rm ht}_S(J),
\end{equation}
where ${\rm ht}_S(J)$ is the \emph{height} of the ideal $J$ in $S$. In our application, the algebraic set $V(J)$ is radical and equidimensional; 
${\rm ht}_S(J)$ is the codimension of the variety $V(J)\subseteq {\mathbb C}^N$.

\begin{example}\label{exa:computed}
Continuing Example~\ref{exa:1}, using Macaulay2's \texttt{resolution} and \texttt{betti} one can compute the Betti numbers for the minimal free resolution of
$T_{1423576,7314562}$ as
\begin{verbatim}
                      0  1  2   3   4   5   6   7  8  9 10
               total: 1 12 61 176 322 392 322 176 61 12  1
                   0: 1  7 21  35  35  21   7   1  .  .  .
                   1: .  5 40 140 280 350 280 140 40  5  .
                   2: .  .  .   1   7  21  35  35 21  7  1
\end{verbatim}
In Macaulay2 format, the entry in row $j$ and column $i$ is $\beta_{i,i+j}$. So
${\rm Reg}({\mathbb C}^{(v)}/T_{1423576,7314562})=2$ is the largest row index of this table. Similarly one checks
that  ${\rm Reg}({\mathbb C}^{(v)}/T_{1234567,7314562})=3$, in agreement with Conjecture~\ref{conj:2}. 
\qed
\end{example}

\section{Proof of Theorem~\ref{thm:main}}\label{sec:4}

\subsection{Proof of Conjectures~\ref{conj:1}, \ref{conj:2}, \ref{conj:3} in the covexillary case} 
Let $R_{v,w}':={\mathbb C}[{\bf z}^{(v)}]/I_{v,w}'$. We claim
\begin{equation}
\label{eqn:firstargue}
{\rm Reg}(R_{v,w}')=\deg H_{v,w}.
\end{equation}

By \cite[Theorems~3.1 and~5.5]{LiYong1}, ${\rm Spec} \, R_{v,w}'$ Gr\"obner degenerates to
${\rm init}_{\prec} {\overline X}_{\kappa(v,w)}$ (up to a permutation of coordinates), the Gr\"obner limit in \cite{KMY} of a matrix Schubert variety
 ${\overline X}_{\kappa(v,w)}$ of the covexillary permutation $\kappa(v,w)$. We will define $\kappa(v,w)$ in Section~\ref{sec:4.2}.
  At this moment, it suffices to know that ${\rm init}_{\prec} {\overline X}_{\kappa(v,w)}$ is a reduced union of coordinate subspaces, whose
 associated Stanley-Reisner simplicial complex is homeomorphic to a shellable ball or sphere \cite[Theorem~4.4]{KMY}. Shellable simplicial complexes
 are Cohen-Macaulay, which by definition, means the said union of coordinate subspaces is Cohen-Macaulay \cite[Section~13.5.3]{Miller.Sturmfels}.
  Therefore
 ${\overline X}_{\kappa(v,w)}$ is Cohen-Macaulay, and hence ${\rm Spec} \, R_{v,w}'$ is also Cohen-Macaulay as it
 also Gr\"obner degenerates to it \cite[Section~15.8]{Eisenbud}. 
 
 In \cite{LiYong1}, one has
 \[{\mathcal K}(R_{v,w}',q)=\frac{H_{v,w}(q)(1-q)^{\ell(w_0w)}}{(1-q)^{\ell(w_0v)}}.\]
 Thus by (\ref{eqn:easyreg2}), ${\rm Reg}(R_{v,w}')=\deg \, H_{w,v}(q) + \ell(w_0w)-\ell(w_0w)$, since ${\rm ht}_{{\mathbb C}[{\bf z}^{(v)}]} {I}_{v,w}'=\ell(w_0w)$ (here we use the fact that the tangent cone of ${\mathcal N}_{v,w}$ has the same dimension as ${\mathcal N}_{v,w}$ itself,
 namely ${\ell}(w)-\ell(v)$, and that ). Thus (\ref{eqn:firstargue}) holds.
 
 Since the tangent cone of ${\mathcal N}_{v,w}$ is ${\rm Spec} \, R'_{v,w}$ it follows from (\ref{eqn:KLfact}) that
 \[\text{tangent cone ($v\Omega_{id}^{\circ}\cap X_w$)}\cong {\rm Spec} \, R'_{v,w}\times {\mathbb A}^{\ell(v)}.\] 
The tangent cone of any affine open neighborhood of $p\in Y$ 
 is isomorphic to $R_{p,Y}$; see, \emph{e.g.}, \cite[Section~5.4]{Eisenbud} and
 \cite[III.3]{Mumford}. Hence the Cohen-Macaulayness of $R'_{v,w}$ implies the same of $R_{v,w}$, since this property
 of an affine variety is preserved under cartesian product with affine space.
  Hence Conjecture~\ref{conj:1} holds in this case by (\ref{eqn:easyreg1}).
 
 Conjecture~\ref{conj:2} holds in our case since it is shown in \cite{LiYong2} that $H_{v,w}(q)$ is semicontinuous. Also, in the covexillary
 case, one has from \emph{ibid.} that $\deg H_{v,w}(q)=\deg P_{v,w}(q)$ where $P_{v,w}(q)$ is the \emph{Kazhdan-Lusztig polynomial}.
 By definition $\deg P_{v,w}(q)\leq \frac{\ell(w)-\ell(v)-1}{2}$; this is Conjecture~\ref{conj:3}.

\subsection{Permutation combinatorics and the formula}\label{sec:4.2}
We recall some standard permutation combinatorics; our reference is
\cite{Manivel} (although our conventions are upside down from theirs). The \emph{graph} of $w\in {\mathfrak S}_n$ places a $\bullet$ in position $(w(i),i)$ (written in matrix notation). Cross out all boxes weakly right and weakly above a $\bullet$; the remaining boxes of $[n]\times [n]$ form the \emph{Rothe diagram} of $w$, denoted $D(w)$. That is,
\[D(w)=\{(i,j)\in [n]\times [n]: i>w(j), j<w^{-1}(i)\}.\]
The vector ${\sf code}(w)=(c_{n},c_{n-1},\ldots,c_1)$ where $c_i$ is the number boxes of $D(w)$ in row $i$.
The \emph{essential set} $E(w)$ of $w$ consists of those maximally northeast boxes of any connected component of $D(w)$, \emph{i.e.},
\[E(w)=\{(i,j)\in D(w): (i-1,j), (i,j+1)\not\in D(w)\}.\]

\begin{example} 
Continuing our running example, where $w=7314562$, diagram is graphically depicted in Figure~\ref{fig:1}. Hence
\[D(w)=\{(2,3),(4,2),(4,3),(5,2),(5,3),(5,4),(6,2),(6,3),(6,4)\}\]
and
\[E(w)=\{{\mathfrak e}_1=(6,5),{\mathfrak e}_2=(5,4),
{\mathfrak e}_3=(4,2), {\mathfrak e}_4=(2,3)\}.\]
Moreover, ${\sf code}(w)=(0,4,3,2,0,1,0)$.

\begin{figure}[h]\label{figure1}
\begin{picture}(180,105)
\put(-10,50){$D(w)=$}
\put(52.5,15){\line(1,0){60}}
\put(52.5,30){\line(1,0){60}}
\put(52.5,15){\line(0,1){15}}
\put(67.5,15){\line(0,1){15}}
\put(82.5,15){\line(0,1){15}}
\put(97.5,15){\line(0,1){15}}
\put(112.5,15){\line(0,1){15}}
\put(100,20){${\mathfrak e}_1$}
\put(52.5,30){\line(1,0){45}}
\put(52.5,45){\line(1,0){45}}
\put(52.5,30){\line(0,1){15}}
\put(67.5,30){\line(0,1){15}}
\put(82.5,30){\line(0,1){15}}
\put(97.5,30){\line(0,1){15}}
\put(85,35){${\mathfrak e}_2$}
\put(52.5,60){\line(1,0){30}}
\put(52.5,45){\line(0,1){15}}
\put(67.5,45){\line(0,1){15}}
\put(82.5,45){\line(0,1){15}}
\put(70,50){${\mathfrak e}_3$}
\put(67.5,75){\line(1,0){15}}
\put(67.5,90){\line(1,0){15}}
\put(67.5,75){\line(0,1){15}}
\put(82.5,75){\line(0,1){15}}
\put(70,80){${\mathfrak e}_4$}
\thicklines
\put(45,7.5){\circle*{4}}
\put(45,7.5){\line(1,0){97.5}}
\put(45,7.5){\line(0,1){97.5}}
\put(60,67.5){\circle*{4}}
\put(60,67.5){\line(1,0){82.5}}
\put(60,67.5){\line(0,1){37.5}}
\put(75,97.5){\circle*{4}}
\put(75,97.5){\line(1,0){67.5}}
\put(75,97.5){\line(0,1){7.5}}
\put(90,52.5){\circle*{4}}
\put(90,52.5){\line(1,0){52.5}}
\put(90,52.5){\line(0,1){52.5}}
\put(105,37.5){\circle*{4}}
\put(105,37.5){\line(1,0){37.5}}
\put(105,37.5){\line(0,1){67.5}}
\put(120,22.5){\circle*{4}}
\put(120,22.5){\line(1,0){22.5}}
\put(120,22.5){\line(0,1){82.5}}
\put(135,82.5){\circle*{4}}
\put(135,82.5){\line(1,0){7.5}}
\put(135,82.5){\line(0,1){22.5}}
\end{picture}
\caption{The diagram and essential set for $w=7314562$.}
\label{fig:1}
\end{figure}
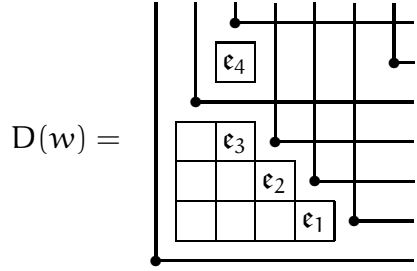
\end{example}

A permutation in ${\mathfrak S}_n$ is uniquely identified by the values of the rank matrix $(r_{ij}^w)$ when restricted to 
$D(w)$ or even merely $E(w)$.

Throughout the remainder of this subsection, we  assume $w$ is covexillary. 

Let $\lambda(w)$ be the partition obtained by sorting ${\sf code}(w)$. It is useful to know the \emph{graphical construction}
of $\lambda(w)$: Since $(a,b),(c,d)\in E(w)$ then one is weakly northwest of the other \cite{Manivel}, it follows
there is a unique Young diagram (in French notation) obtained by pushing all boxes of $D(w)$ 
on a given antidiagonal to the southwest; that is the diagram of $\lambda(w)$.

\begin{example}
Our running example $w=7314562$ is covexillary with $\lambda(w)=(4,3,2,1)$. \qed
\end{example}

Given $v\leq w$, \cite{LiYong1} defines (and proves the existence of) a different covexillary permutation $\kappa(v,w)$. This is the unique permutation whose essential set is obtained by moving each ${\mathfrak e}=(i,j)\in E(w)$ southwest along its antidiagonal by $r^v_{ij}$ squares to ${\mathfrak e}'$ and imposing that $r_{{\mathfrak e}'}^{\kappa(v,w)}=r^w_{ij}-r^v_{ij}$.
By construction, $\lambda(w)=\lambda(\kappa(v,w))$. The graphical construction $\lambda(\kappa(v,w))$ induces a bijection of boxes:
$\phi:\lambda(\kappa(v,w)) \to D(\kappa(v,w))$. Define a filling of each box $b\in \lambda(\kappa(v,w))$ with $r^w_{\phi(b)}$. We call
this ${\sf RRW}(v,w)$, as its provenance is from \cite{RRW}.

\begin{example}
One can check that $\kappa(1423576,7314562)=3472561$.\qed
\end{example}

The next result is the combinatorial rule of Theorem~\ref{thm:main}. It uses a similar result of
J.~Rajchgot-C.~Robichaux-A.~Weigandt \cite[Theorem~1.3]{RRW}:
\begin{theorem} 
\label{thm:formula}
\begin{equation}
\label{eqn:theform}
{\rm Reg}(R_{v,w})={\rm Reg}(R_{v,w}')=\deg H_{v,w}=\sum_{k\geq 1}\sum_{\alpha\in{\sf Connected}(\lambda(\kappa(v,w))_{\geq k})} {\sf maxdiag}(\alpha),
\end{equation}
where:
\begin{itemize}
\item $\lambda(\kappa(v,w))_{\geq k}$ is the shape of the subtableau of ${\sf RRW}(v,w)$ that have entries $\geq k$;
\item ${\sf Connected}(\kappa(v,w))_{\geq k})$ are the connected components of the aforementioned shape; and
\item ${\sf maxdiag}(\alpha)$ is the largest northwest-southeast diagonal that appears in $\alpha$.
\end{itemize}
\end{theorem}

\begin{example}
To complete our running example, 
\[{\sf RRW}(1423576,7314562)=\tableau{0\\ 0 & 0 \\ 0 & 0 &1\\ 0& 0&1 & 1}\]
and hence Theorem~\ref{thm:formula} asserts ${\rm Reg}=2$ (the longest diagonal appearing in the unique $1$'s component), in agreement with Example~\ref{exa:computed}.\qed
\end{example}

For any $u\in {\mathfrak S}_n$ let ${\mathfrak G}_w(x_1,\ldots,x_n)$ be the \emph{Grothendieck polynomial} \cite{LS:groth}. By definition,
${\mathfrak G}_{w_0}=x_1^{n-1}x_2^{n-2}\cdots x_{n-1}$ where $w_0$ is the longest element in ${\mathfrak S}_n$. If $\ell(us_i)>\ell(u)$
where $s_i=(i \ i+1)$ is a simple transposition, then ${\mathfrak G}_{u}=\pi_i({\mathfrak G}_{us_i})$ where 
\[\pi_i:{\mathbb Z}[x_1,x_2,\ldots,x_n]\to {\mathbb Z}[x_1,x_2,\ldots,x_n]\] 
is the \emph{isobaric divided
difference operator} defined by
\[\pi(f)=\frac{(1-x_{i+1})f(\cdots,x_i,x_{i+1},\cdots)-(1-x_i)f(\cdots,x_{i+1},x_i,\cdots)}{x_i-x_{i+1}}.\]

\subsection{Proof of Theorem~\ref{thm:formula}}  By \cite[Theorem~6.6]{LiYong1},
\begin{equation}
\label{eqn:finalPS}
{\sf PS}_{v,w}(q)=\frac{G_{\lambda}(q)}{(1-q)^{n\choose 2}},
\end{equation}
where $G_{\lambda}(q)={\mathfrak G}_{w_0\kappa(v,w)}(1-q,1-q,\ldots,1-q)$. Comparing (\ref{eqn:finalPS}) with (\ref{eqn:firstPS})
and using the fact that $\dim(X_w)=\ell(w)$, we see that
\begin{equation}
\label{eqn:May25ttt}
\deg \, H_{v,w} = \deg \, {\mathfrak G}_{w_0\kappa(v,w)} -\left ({n\choose 2}-\ell(w)\right).
\end{equation}
On the other hand, since $\lambda(\kappa(v,w))=\lambda(w)$, one has $\ell(\kappa(v,w))=\ell(w)$, and hence 
\begin{equation}
\label{eqn:May25uuu}
\ell(w_0\kappa(v,w))={n\choose 2}-\ell(w).
\end{equation}
Moreover since $\kappa(v,w)$ is covexillary, $w_0\kappa(v,w)$ is vexillary (avoids $2143$).  The formula of 
J.~Rajchgot-C.~Robichaux-A.~Weigandt \cite[Theorem~1.3]{RRW} shows (in our conventions) that for any vexillary 
$u\in {\mathfrak S}_n$ that
\begin{equation}
\label{eqn:RRWformula123}
\deg {\mathfrak G}_u=\ell(u)+\sum_{k\geq 1}\sum_{\alpha\in{\sf Connected}(\lambda(w_0u)_{\geq k})} {\sf maxdiag}(\alpha).
\end{equation}
Hence the theorem follows by combining (\ref{eqn:May25ttt}), (\ref{eqn:May25uuu}) and (\ref{eqn:RRWformula123}) with $u=w_0\kappa(v,w)$.\qed

In general, there are no simple formulas to compute the degree of a Kazhdan-Lusztig polynomial $P_{v,w}(q)$ (we refer the reader to
\cite[Chapter~5]{Bjorner.Brenti}). This proves the final assertion of Theorem~\ref{thm:main}.

\begin{corollary}
Let $w\in {\mathfrak S}_n$ be covexillary, then $\deg\, P_{u,v}$ is computed by the rule of Theorem~\ref{thm:formula}.
\end{corollary}
\begin{proof}
\cite[Theorem~1.2]{LiYong2} shows 
$\deg\, H_{v,w}(q)=\deg\, P_{v,w}(q)$ when $w$ is covexillary. Now apply Theorem~\ref{thm:formula}.
\end{proof}
\section{Further results and discussion}\label{sec:5}

These conjectures were asserted in~\cite{LiYong2}:

\begin{conjecture}\label{conj:1'}
$R_{v,w}$ is Cohen-Macaulay. Consequently, $H_{v,w}\in {\mathbb N}[q]$.
\end{conjecture}
 
That $X_w$ is Cohen-Macaulay does not imply Conjecture~\ref{conj:1'}. In fact, C.~Huneke \cite{Huneke} established $R_{p,Y}$ being
Cohen-Macaulay implies the same for $({\mathcal O}_{p,Y},{\mathfrak m}_p,\Bbbk)$, and gave counterexamples for the converse.
This is a strengthening of Conjecture~\ref{conj:1'}:
\begin{conjecture}[Semicontinuity]\label{conj:2'}
If $u\leq v\leq w$ then $[q^t]H_{u,w}\geq [q^t]H_{v,w}$.
\end{conjecture}

\begin{conjecture}[{\cite[Proposition~2.1]{LiYong2}}]
\label{conj:3'} 
$\deg H_{v,w}\leq \frac{\ell(w)-\ell(v)-1}{2}.$
\end{conjecture}

\begin{proposition}\label{evenstronger}
Conjectures~\ref{conj:1'},~\ref{conj:2'}, and~\ref{conj:3'} imply Conjectures~\ref{conj:1},~\ref{conj:2}, and~\ref{conj:3}.
\end{proposition}
\begin{proof}
The Cohen-Macaulay assertion of Conjecture~\ref{conj:1'} implies Conjecture~\ref{conj:1} by the reasoning in our proof of Theorem~\ref{thm:main}. 
Combined with Conjecture~\ref{conj:2'} gives Conjecture~\ref{conj:2}. Separately, combined with Conjecture~\ref{conj:3'} one
would obtain Conjecture~\ref{conj:3}.
\end{proof}

During the preparation of \cite{LiYong2}, Conjectures~\ref{conj:1'} and~\ref{conj:3'} were checked for $n\leq 7$. Conjecture~\ref{conj:2'} 
was checked for at least $n\leq 6$ and much of $n=7$. 

Let
${\rm maxReg}(n)=\max_{v\leq w\in {\mathfrak S}_n}{\rm Reg}(R_{v,w})$.

\begin{conjecture}
\label{conj:maxReg}
${\rm maxReg}(n)=\Theta(n^2)$.
\end{conjecture}

Computational data was not directly useful to arrive at Conjecture~\ref{conj:maxReg}. For $n=4,5,6,7$,
${\rm maxReg}(n)=1,2,3,5$, respectively. For example, when $n=7$ the maximizer
is the (non-covexillary) $w=6734512$ at $v=id$. Here $I_{v,w}$ is inhomogeneous and 
\[H_{id,6734512}(q)=1+4q+9q^2+9q^3+4q^4+q^5.\]

Let 
$\overline{{\rm maxReg}(n)}=\max_{v\leq w\in {\mathfrak S}_n, w\text{\ covexillary}}{\rm Reg}(R_{v,w})$. 
We apply Theorem~\ref{thm:formula} to prove the covexillary case of Conjecture~\ref{conj:maxReg}.

\begin{proposition}
\label{prop:covexmax}
$\overline{{\rm maxReg}(n)}=\Theta(n^2)$.
\end{proposition}
\begin{proof}
For the lower bound, first suppose $n=3j-1$ for $j\geq 1$. Let $v=id$ and $w\in {\mathfrak S}_{n}$ be the unique permutation 
with ${\sf code}(w)=(1,2,3,\ldots,j,0,0,\ldots,0)$. Then $w$ is covexillary, with $\lambda(w)=(j,j-1,\ldots,3,2,1)$. For example, if $j=4$ then
$w=7,11,6,10,5,9,4,8,3,2,1$. By our assumption, $\kappa(id,w)=w$. Hence ${\sf RRW}(\kappa(id,w))$ is the staircase $\lambda(w)$
where column $c$ from the left is filled by $(c-1)$'s. In our example,
\[{\sf RRW}(\kappa(id,w))=\tableau{0 \\ 0 & 1\\ 0 & 1 & 2 \\ 0 & 1 & 2 & 3}.\]
Hence, Theorem~\ref{thm:formula} asserts that ${\rm Reg}(R_{id,w})=(j-1)+(j-2)+\ldots+2+1={j\choose 2}$. Now, if $n=3j$ or $n=3j+1$,
use the same construction as for $n=3j-1$, except that ${\sf code}(w)$ will have an additional $0$ or $0,0$ postpended, respectively.
In those two cases, the same analysis implies ${\rm Reg}(R_{id,w})={j\choose 2}$. Hence $\overline{{\rm maxReg}(n)}=\Omega(n^2)$ follows.

For the upper bound, since $w\in {\mathfrak S}_n$, $\lambda(\kappa(v,w))\subseteq n\times n$ and ${\sf RRW}(\kappa(v,w))$ only uses
labels $k\in [n]$. For each such $k$, the inner sum  of (\ref{eqn:theform}) contributes $\leq n$. Hence 
${\rm Reg}(R_{v,w})\leq n^2$. Therefore, $\overline{{\rm maxReg}(n)}=O(n^2)$, as required.
\end{proof}

\begin{corollary}
Conjecture~\ref{conj:3} implies Conjecture~\ref{conj:maxReg}.
\end{corollary}
\begin{proof}
The lower bound of Conjecture~\ref{conj:maxReg} is immediate from Proposition~\ref{prop:covexmax}. If Conjecture~\ref{conj:3} holds, then
${\rm Reg}(R_{v,w})\leq \frac{\ell(w)-\ell(v)-1}{2}\leq \ell(w_0)={n\choose 2}$.
\end{proof}

Sometimes, $I_{v,w}$ is homogeneous with respect to the standard grading; see \cite{WY:Amer}
and the references therein. In those cases, trivially, ${I}_{v,w}'=I_{v,w}$ and Cohen-Macaulayness of $I_{v,w}$ and Conjecture~\ref{conj:1} is automatic. As argued in \cite{LiYong2}, the covexillary case is interesting precisely because 
${I}_{v,w}'=I_{v,w}$ need not hold in general (as in the case of our running example).

It is also natural to expect that our regularity conjectures are true for other Lie types. We remark that in the minuscule case studied by \cite{GK}, it is again true that the Schubert varieties admit a dilation action of ${\mathbb C}^*$ and hence the analogue
of Conjecture~\ref{conj:1} holds for a similar reason as in the previous paragraph. This problem should be in reach:

\begin{problem}
Determine the regularity of tangent cones of Schubert varieties for minuscule $G/P$.
\end{problem}

We also mention that the \emph{banner permutations} of Z.~Hamaker-O.~Pechenik-A.~Weigandt \cite{HPW} extend the vexillary permutations and
have a description of the Gr\"obner basis (also, see a further extension by P.~Klein~\cite{Klein}). It would therefore be interesting to see if the results of this paper (or of \cite{LiYong1,LiYong2})
extend to that setting.

With regards to Theorem~\ref{thm:formula}, one can use any rule that computes ${\rm deg}({\mathfrak G}_u)$. Another rule applicable to arbitrary $u\in {\mathfrak S}_n$ has been found by O.~Pechenik-D.~Speyer-A.~Weigandt
\cite{PSW}. On the one hand, the tableau rule of 
\cite{RRW} is fitting with the covexillary combinatorics we use. On the other hand, one wonders if that general rule can be adapted
to compute ${\sf Reg}(R_{u,v})$? We also remark that both of these formulas can be regarded as solving a special case of our
regularity problem; see \cite[Corollary 2.6]{WY:Amer} and its proof.

\section*{Acknowledgements}
We thank Daniel Erman, Shiliang Gao, Cao Huy Linh, Jenna Rajchgot, Hal Schenck, Colleen Robichaux, Anna Weigandt, and Alexander Woo for helpful communications. We are grateful to Allen Knutson, Li Li, Ezra Miller and Alexander Woo for joint work on which
this piece is based. We also thank the organizers of ``Singularities of the Midwest (online edition)'' and ``Recent Developments in Gr\"obner Geometry''; these notes are based on the talks given at those Summer '21 events. We made use of Macaulay2 in our investigations.
AY was partially supported by a Simons Collaboration Grant, an NSF RTG 1937241 in Combinatorics, and an appointment at the UIUC Center for Advanced Study.

\end{document}